\documentclass[Roman,12pt,onecolumn,A4]{article}
\usepackage{amsmath}
\usepackage{amsthm}
\usepackage{amssymb}
\usepackage{amsfonts}
\usepackage{fancyhdr}
\usepackage{array}
\usepackage{pst-plot}
\usepackage{tikz}
\usepackage{tkz-euclide}
\usetkzobj{all}
\usepackage[left=.5in,right=.5in,top=1in,bottom=.75in]{geometry}
\usepackage{enumerate}
\usepackage{ragged2e}
\usepackage{esint}
\title{Research Papers}
\author{James Tipton}
\newtheorem{thm}{Theorem}
\newtheorem{ex}[thm]{Example}

\newtheorem{prop}[thm]{Proposition}

\newcommand{\C}[0]{\mathbb{C}}

\title{Partial Classification of Polynomials and an Orthonormal Basis Construction on the Associated Basin of Attraction}

\usepackage{color,soul}

\begin{document}

\maketitle

\abstract{In the paper \textit{Infinite product representations for kernels and iterations of functions}, the authors associate certain Fatou subsets with reproducing kernel Hilbert spaces.  They also present a method for constructing an orthonormal basis for said Hilbert space, but the method depends on the polynomial of the given Fatou set. We provide a partial classification of those polynomials the method applies to.

\section{Introduction}

\subsection*{Complex Dynamics}

Recall that $R:\mathbb{C}\rightarrow\mathbb{C}$ has an attracting fixed point at $z_0\in\mathbb{C}$ if $|R^{'}(c)|<1$.  The point $z_0$ is called an attracting fixed point because all points within a certain neighborhood of $z_0$ are ``attracted'' to $z_0$ under repeated iteration of $R$.  The $n$th iterate of $R$ is denoted by
\[R^{\circ n}(z) = \underbrace{R\circ R\circ\cdots \circ R(z)}_{n\text{ times}}\]
The basin of attraction of $R$ at the attracting fixed point $z_0$ is the following subset of $\mathbb{C}$:
\[B_{R,z_0}=\{z\in\mathbb{C}:\lim_{n\rightarrow\infty}R^{\circ n}(z)=z_0\}\]
For many polynomials with an attracting fixed point, the basin of attraction is a fractal.\\

\subsection*{Reproducing Kernel Hilbert Spaces}

A reproducing kernel Hilbert space (RKHS) on $\mathbb{C}$ is a Hilbert space of functions on $\mathbb{C}$ in which every linear evaluation functional is bounded.  Uniquely associated to each RKHS is a kernel function $K : \C\times\C\rightarrow\C$ with the reproducing property:
\[\langle f(z), K(z,w) \rangle_{\mathcal{H}} = f(w)\]
Since a RKHS is, in particular, a Hilbert space, it must have an orthonormal basis (ONB).  Although ONBs are guaranteed to exist, explicitly constructing an ONB is a much harder task.  %\hl{The next three sections summarize key results from} \cite{} \hl{unless noted otherwise.}

\subsection*{Kernel Functions on Basins of Attraction}

If $R$ satisfies sufficient conditions, then one may construct a kernel function, represented as an infinite product, on a subset of $B_{R,z_0}$.  See [1] for the general result. In particular, if $R$ is a polynomial and $z_0 = 0$, then the map $K:\mathbb{C}\times\mathbb{C}\rightarrow\mathbb{C}$ defined by
\[K(z,w) = \prod_{n=0}^\infty\left(1+R^{\circ n}(z)\overline{R^{\circ n}(w)}\right)\tag{1}\]
is a kernel function on all of $B_{R,0}$ [3].  The infinite product involves iterates of the map $R$ and the map $1+z\overline{w}$, which is a kernel function on $\C$.  The kernel function $1 + z\overline{w}$ can be used to construct an ONB under certain circumstances.

\subsection*{The ONB Construction}
First we take a moment to recall multi-index notation.  Suppose $J$ is an index set, then
\[J^\infty = \{v: v\in J^N\text{ for some }N=1,2,\dots\}\]
Denote the RKHS associated to the previous kernel function $(1)$ by $\mathcal{H}$.  The constant function $\mathbf{1}(z) = 1$ plays a crucial role in the construction, and in fact belongs to $\mathcal{H}$.  Consider a family of operators on $\mathcal{H}$, $\{S_i:\mathcal{H}\rightarrow\mathcal{H}\}$.  For each $v=(v_1,v_2,\dots,v_N)\in J^N$, define $b_v:\mathbb{C}\rightarrow\mathbb{C}$
\[b_v(z)=(S_{v_1}S_{v_2}\cdots S_{v_N}\mathbf{1})(z)\]
The next theorem, due to the authors of [1], gives sufficient conditions for the functions $b_v$ to form an ONB.
\begin{thm}
If a family of operators $\{S_i:\mathcal{H}\rightarrow\mathcal{H}\}_{i=1}^N$ satisfies the Cuntz relations:
\[S_i^{*}S_j = \delta_{ij}I,\quad\sum_{i=1}^N S_iS_i^{*}=I\]
then $B = \{b_v: v\in J^\infty\}$ is an ONB for $\mathcal{H}$.
\end{thm}

In our particular set-up, the family we wish to consider is $\mathcal{F}=\{S_1, S_2\}$ where $S_1f(z) = f(R(z))$ and $S_2f(z)=zf(R(z))$.  This family of operators can be shown to satisfy the Cuntz relations when certain conditions are met, which we discuss now.

\subsection*{The Dagger Conditions}
%The conditions we wish to discuss have been modified slightly from those given in \cite{}.  In particular, we require the conditions to factor in the multiplicity of solutions to $R(\zeta) = z$.  The theory presented in \cite{} follows in exactly the same fashion with this modification.\\

The family $\mathcal{F}$ of interest depends on the map $R$ that is chosen.  It can be shown that if $R$ satisfies for all $z\in B_{R,0}$,
\[M(z)<\infty\]
where $M(z)$ is the number of solutions to $R(\zeta)=z$, counting multiplicity, and either
\begin{align*}
\frac{1}{M(z)}\sum_{R(\zeta)=z}e_i(\zeta)\overline{e_j(\zeta)}=\delta_{ij},\qquad \forall i,j\in J\tag{$\dagger$}\\
\text{or}\hspace{1.5in}\\
\frac{1}{M(z)}\sum_{R(\zeta)=z}e_i(\zeta)e_j(\zeta)=\delta_{ij},\qquad \forall i,j\in J\tag{$\ddagger$},
\end{align*}
then $\mathcal{F}$ satisfies the Cuntz relations [1].  The functions, $e_i(z)$, are taken from any ONB for the RKHS associated to the underlying kernel function of the infinite product kernel function.  For the family $\mathcal{F}$ that we are interested in, we have that $e_1(z) = 1$ and $e_2(z) = z$; this comes from the underlying kernel function $1+z\overline{w}$ mentioned earlier.  For ease of exposition we will refer to the above conditions as the dagger conditions.  A natural question is when does the map $R$ satisfy either of the above conditions?\\

In the context of the underlying kernel function $1 + z\overline{w}$, the $\dagger$ condition becomes
\begin{align*}
\sum_{R(\zeta)=z}1 &= M(z) \tag{$\dagger_1$}\\
\sum_{R(\zeta)=z}\zeta &= 0 = \sum_{R(\zeta)=z} \tag{$\dagger_2$}\overline{\zeta}\\
\sum_{R(\zeta)=z}|\zeta|^2 &= M(z) \tag{$\dagger_3$},
\end{align*}
and the $\ddagger$ condition becomes
\begin{align*}
\sum_{R(\zeta)=z}1 &= M(z) \tag{$\ddagger_1$}\\
\sum_{R(\zeta)=z}\zeta &= 0 \tag{$\ddagger_2$}\\
\sum_{R(\zeta)=z}\zeta^2 &=M(z) \tag{$\ddagger_3$}
\end{align*}

We examine which polynomials $R$ satisfy the dagger conditions and offer a classification for $R$ to satisfy the $\ddagger$ condition.

\section{Partial Classification of the Dagger Conditions}

The purpose of the dagger conditions is to construct an ONB for the RKHS corresponding to the kernel function on $B_{R,0}$.  Thus our interest lies only with those polynomials with an attracting fixed point at $0$, even though the dagger conditions do not require $R$ to have such a property.  The first two cases of either dagger condition is quite easily characterized.

\begin{prop}
If $P(z)$ is a degree $n$ polynomial with an attracting fixed point at $0$ then the following hold:
\begin{enumerate}[a)]
\item $P(z)$ satisfies $\dagger_1$ and $\ddagger_1$.
\item $P(z)$ satisfies $\dagger_2$ and $\ddagger_2$ if and only if $a_{n-1} = 0$.
\end{enumerate}
\end{prop}
\begin{proof}
\noindent
\begin{enumerate}[a)]
\item Since $B_{P,0}$ is completely invariant with respect to $P$, we know that if $P(\zeta) = z$ for some $z\in B_{P,0}$, then we must have that $\zeta\in\Omega$.  By the Fundamental Theorem of Algebra there are $n$ solutions to $P(\zeta) = z$, counting multiplicity.  Thus $\sum\limits_{P(\zeta)=z}1 = n = M(z)$.

\item We have that $P(z)$ satisfies $\dagger_2$ and $\ddagger_2$ if and only if $\sum\limits_{P(\zeta)=z}\overline{\zeta} = 0 = \sum\limits_{P(\zeta)=z}\zeta$ which by Vieta's formulas is equivalent to $a_{n-1} = 0$.
\end{enumerate}
\end{proof}

\noindent\textbf{Note:} The conditions $\dagger_2$ and $\ddagger_2$ are equivalent to each other since $\sum\limits_{P(\zeta)=z}\overline{\zeta} = \overline{\sum\limits_{P(\zeta)=z}\zeta}$

The next proposition will complete our characterization of $\ddagger$.

\begin{prop}
A polynomial $P$ of degree $n\geq3$, with an attracting fixed point at $0$, satisfies $\ddagger_3$ if and only if $a_{n-2} = \dfrac{-na_n}{2}$
\end{prop}

\begin{proof}
The polynomial $P$ satisfies $\ddagger_3$ if and only if
\[\sum_{R(\zeta)=z}\zeta^2 = n\]
which is equivalent to $a_{n-2} = \frac{-na_n}{2}$, since Vieta's formulas and the Newton-Girard formulas give
\[n = \sum_{R(\zeta)=z}\zeta^2 = -2e_2 = \frac{-2a_{n-2}}{a_n}\]
\end{proof}

Combining the last two propositions, we may characterize those polynomials satisfying the $\ddagger$ condition:

\begin{thm}
Suppose $P$ is a polynomial of degree $n\geq 3$ with an attracting fixed point at $0$.  The polynomial $P$ satisfies $\ddagger$ if and only if $P(z) = \sum\limits_{k=1}^na_kz^k$ where $a_{n-1}=0$, and $a_{n-2}=\frac{-na_n}{2}$.
\end{thm}

\noindent\textbf{Note:} Since $P$ has an attracting fixed point at $0$, we have also that $|a_1|<1$.\\

Characterizing the $\dagger_3$ condition seems to be more challenging.  It is easy to find polynomials which satisfy $\dagger_3$ at a particular point.  But the dagger conditions are required to hold for all $z\in B_{P,0}$.  As such we introduce ``partial'' conditions, $\dagger^c$ and $\ddagger^c$, each meaning that the corresponding set of equations hold precisely at the point $c\in\mathbb{C}$.  The next proposition will help us determine which polynomials might satisfy $\dagger$.

\begin{prop}
Suppose $c\in B_{P,0}$.
\begin{enumerate}[a)]
\item If $P(z)$ satisfies $\ddagger^c$, then $P(z)$ satisfies $\ddagger$.
\item If $P(z)$ satisfies both $\dagger^c$ and $\ddagger^c$ then the equation $P(z)=c$ has only real solutions.
\end{enumerate}
\end{prop}

\begin{proof}
\noindent\\
\begin{enumerate}[a)]
\item If $P(z)$ satisfies $\ddagger^c$, then we must have that
\[\sum_{P(\zeta)=c}1 = n,\quad \sum_{P(\zeta)=c}\zeta = 0,\quad \sum_{P(\zeta)=c}\zeta^2 = n\]
By Vieta's formulas, we have that $a_{n-1} = 0$.  Thus we have that $\ddagger_1$ and $\ddagger_2$ hold by an application of Proposition 1.  By an application of the Newton-Girard formulas we have that
\[n = \sum_{P(\zeta)=c}\zeta^2 = -2e_2 = \frac{-2a_{n-2}}{a_n}\]
Applying the same formula again we find that
\[\sum_{P(\zeta)=z}\zeta^2 = \frac{-2a_{n-2}}{a_n} = n\]
So $P$ actually satisfies $\ddagger$.

\item Suppose that $P(z)$ satisfies both $\dagger^c$ and $\ddagger^c$, so in particular, we have that
\[\sum_{R(\zeta)=c}\zeta^2 = \sum_{R(\zeta)=c}|\zeta^2|\]
which requires that $\sum\limits_{R(\zeta)=c}\text{Im }\zeta^2 = 0$.  Thus we have that
\[\sum_{R(\zeta)=c}\text{Re }\zeta^2 = \sum_{R(\zeta)=c}|\zeta^2|\]
If $\zeta^2$ is not real, then $|\zeta^2|>\text{Re }\zeta^2$, so the above equality holds only if all solutions to $P(\zeta)=c$ are real.
\end{enumerate}
\end{proof}

So if $P(z)$ has a non-real zero, then $P(z)$ cannot satisfy both $\dagger$ and $\ddagger$.  The only interesting polynomials which might satisfy $\dagger$ are those $P(z)$ which have the property that $\ddagger^c$ is not satisfied for any $c\in\Omega$.  Such a polynomial requires the property that $P(z)+c$ has at least one non-real zero for all $c\in\Omega$.  An example of a polynomial with this property is any cubic of the form $P(z) = az^3 + bz + c$ with $b$ a scalar multiple of $\overline{a}$.  We show now that no polynomial can satisfy both $\dagger$ and $\ddagger$.  This fact can be used in turn to show that the aforementioned property is a necessary condition for a polynomial to satisfy $\dagger$.\\

\begin{prop}
If $P$ is a polynomial with an attracting fixed point at $0$, then $P$ cannot satisfy both $\dagger$ and $\ddagger$.
\end{prop}

\begin{proof}
Suppose to the contrary that $P$ satisfies both $\dagger$ and $\ddagger$.  By Proposition 5, $P(z)=c$ has only real solutions, for any $c\in B_{P,0}$.  Thus the inverse image of $B_{P,0}$ under $P$ must be a subset of $\mathbb{R}$. However, $B_{P,0}$ is an open set, and $P$ is a continuous map, so that the inverse image of $B_{P,0}$ under $P$ must be an open set.  But no subset of $\mathbb{R}$ is open as a subset of $\mathbb{C}$.  Thus it cannot be that $P$ satisfies both $\dagger$ and $\ddagger$.
\end{proof}

\noindent We can now state a necessary condition for a polynomial to satisfy $\dagger$.

\begin{prop}
If $P$ is a polynomial with an attracting fixed point at $0$ that satisfies $\dagger$, then for all $c\in B_{R,0}$, the equation $P(z) = c$ has at least one non-real solution.
\end{prop}

\begin{proof}
Suppose that $P$ satisfies $\dagger$, and to the contrary, that $P(z) = c$ has only real solutions.  If $\zeta$ is such a solution, then $|\zeta|^2 = \zeta^2$.  Since $P$ satisfies $\dagger_3$, we have also that, by the previous observation, that $P$ satisfies $\ddagger_3$.  By Proposition 2, $P$ also satisfies $\ddagger_1$ and $\ddagger_2$, so that $P$ must satisfy $\ddagger$.  But $P$ cannot satisfy both $\dagger$ and $\ddagger$, so that $P(z) = c$ must have at least one non-real solution.
\end{proof}

Another approach to showing that a polynomial satisfies one of the dagger conditions is to write it as a product of two polynomials, each of which satisfy the same dagger condition.  This is equivalent to determining whether the product of two polynomials, both satisfying the same dagger condition, will satisfy a dagger condition. This works quite well for the $\ddagger$ condition.

\begin{prop}
If $R(z)$ and $Q(z)$ satisfy $\ddagger$, then $R(z)Q(z)$ satisfies $\ddagger$.
\end{prop}

\begin{proof}
Suppose that $R$ and $Q$ both satisfy $\ddagger$ and let $S(z) = R(z)Q(z)$.  Let $a_i$ denote the coefficients of $R$, $b_i$ denote the coefficients of $Q$, and suppose $\deg R = r$, $\deg Q = q$ so that $\deg S = r + q$.  Theorem 4 tells us that $a_{r-2}=\dfrac{-ra_r}{2} $, $b_{q-2}=\dfrac{-rb_q}{2}$, and $a_{r-1}=0=b_{q-1}$. If $c_i$ denotes the coefficients of $S$, then we have that $c_{r+q}=a_rb_q$, $c_{r+q-1}=0$, and $c_{r+q-2}=\dfrac{-(r+q)c_{r+q}}{2}$.  By Theorem 4, the polynomial $S$ satisfies the $\ddagger$ condition.
\end{proof}

\section{Examples}
Here we present an example of a polynomial satisfying $\ddagger$, and an example of a polynomial satisfying $\dagger^0$.

\begin{ex}
Consider the polynomial $R(z)=iz^4 - 2iz^2 - \frac{1+i}{2}z$.  We see that $0$ is an attracting fixed point of $R$ since $R(0) =0 $ and $|R^{'}(0)|=\dfrac{\sqrt{2}}{2}<1$.  So we have that the map
\[K(z,w) = \prod_{n=0}^\infty \Big(1+R^{\circ n}(z)\overline{R^{\circ n}(w)}\Big)\]
is a kernel function on $B_{R,0}$.  The polynomial $R$ has coefficients:  $a_4 = i$, $a_2 = -2i$, $a_1=-\frac{1+i}{2}$, and $a_3=a_0=0$.  Since $a_3 = 0$ and $ a_2 = \frac{-4 a_4}{2}$, Theorem 4 tells us that $R$ satisfies $\ddagger$.  This in turn shows that the operators $S_1$ and $S_2$, defined by:
\[S_1f(z) = f(R(z))\qquad\text{and}\qquad S_2f(z)=zf(R(z))\]
satisfy the Cuntz relations.  So we may apply Theorem 1 to conclude that the functions $b_v(z)$ form an ONB for the Hilbert space associated to $K$.  Recall that
\[b_v(z) = (S_{v_1}S_{v_2}\cdots S_{v_N}\mathbf{1})(z)\]
where $v\in J^\infty$.  The first few basis elements are:
\[1,z,R(z),zR(z), R^{\circ 2}(z), zR^{\circ 2}(z), R(z) R^{\circ 2}(z), zR(z) R^{\circ 2}(z), \dots\]
So the basis elements may be calculated recursively, but obtaining a general formula appears to require a general formula for $R^{\circ n}$.
\begin{figure}[t]
\centering
\includegraphics[scale=.15]{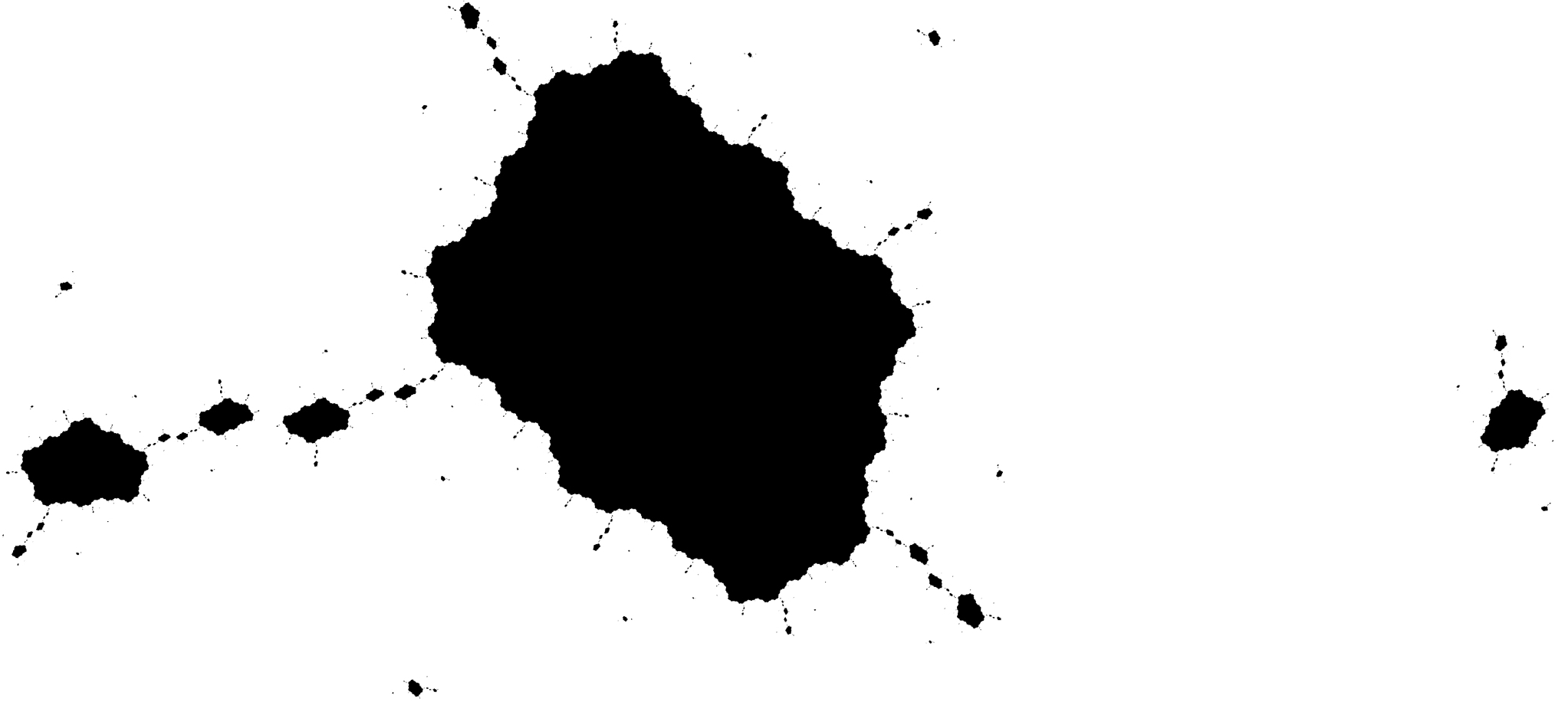}
\caption{The domain of K(z,w): $B_{R,0}$}
\end{figure}
\end{ex}

\begin{ex}
Consider the polynomial $Q(z) = \frac{1}{2}z^3 + \frac{3}{4}z$ which also has an attracting fixed point at $0$.  So the map
\[K(z,w) = \prod_{n=0}^\infty \Big(1+Q^{\circ n}(z)\overline{Q^{\circ n}(w)}\Big)\]
is a kernel function on $B_{Q,0}$.  There is a RKHS associated to $K$, however, we can't use the dagger conditions to construct an ONB.  By Theorem 4, $Q$ doesn't satisfy $\ddagger$, in particular, the condition $\ddagger_3$.  It turns out that $Q$ does satisfy $\dagger^0$; this follows from Proposition 2 and the following observation:
\[\sum_{Q(\zeta)=0}|\zeta|^2 = \Big|0\Big|^2 + \Big|i\sqrt{3/2}\Big|^2 + \Big|-i\sqrt{3/2}\Big|^2=3\]
It can be shown in a similar fashion that $Q$ satisfies $\dagger^c$ for $c = \dfrac{i}{2\sqrt{2}}$.  However $Q$ doesn't satisfy $\dagger^c$ for all $c\in B_{Q,0}$: for $c=i\in B_{Q,0}$ and with the aid of WolframAlpha, we have that
\[\sum_{Q(\zeta)=i}|\zeta|^2 >3\]
 
\begin{figure}[t]
\centering
\includegraphics[scale=.05]{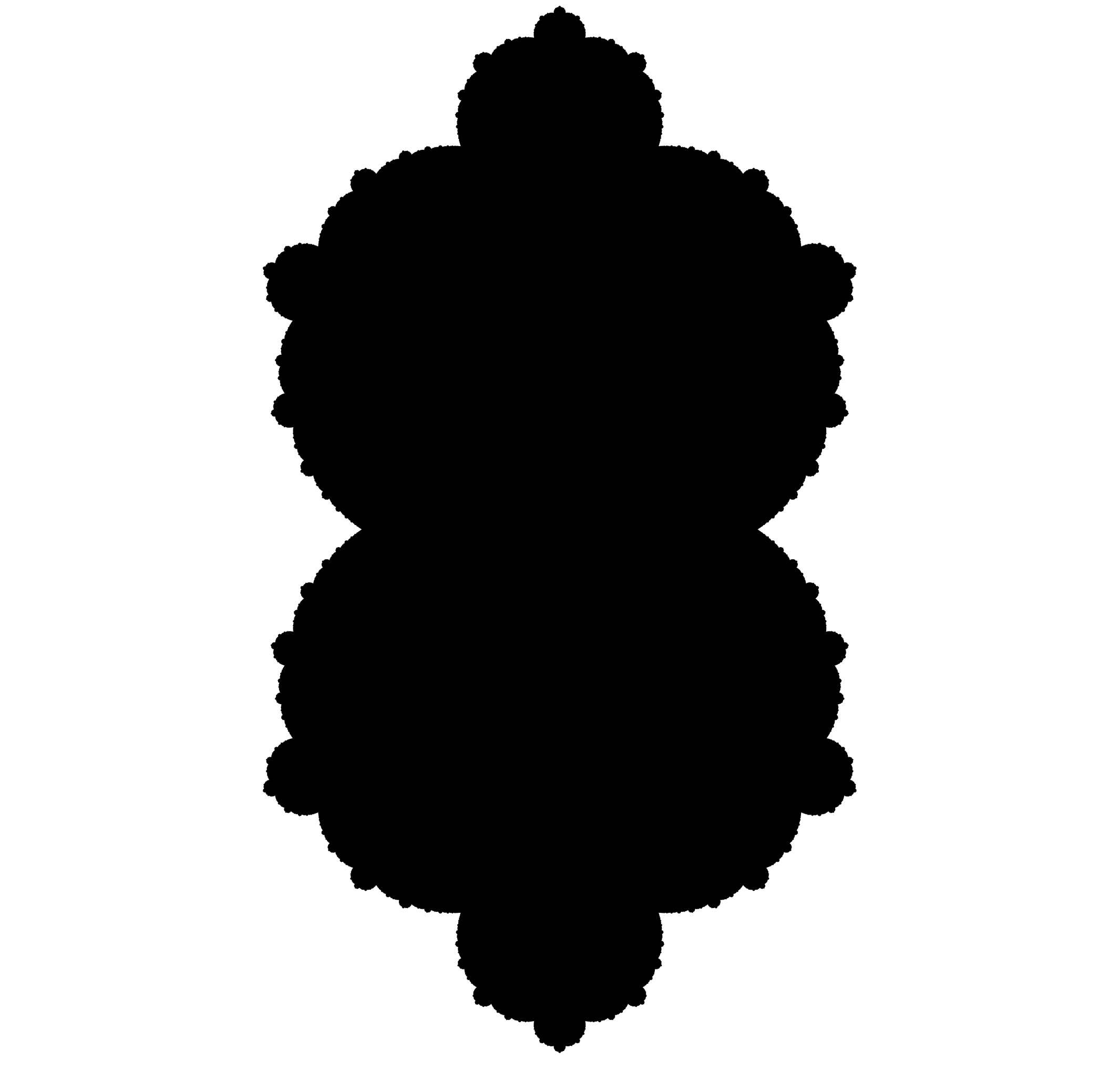}
\caption{The domain of K(z,w): $B_{Q,0}$}
\end{figure}
\end{ex}

\noindent

\section{Open Questions}
Here we discuss some open questions pertaining to this paper.

\begin{enumerate}
\item  \textbf{Classify the } $\dagger$ \textbf{condition}.  The main issue lies with $\dagger_3$.  Since the sum involves the modulus of the roots, Vieta's formula may not be applied.
\item \textbf{Find a polynomial that satisfies} $\dagger_3$ \textbf{or show that no polynomial satifies} $\dagger_3$.  Just having one example would be a nice starting point; but so would knowing that no examples exist. 
\item \textbf{What if we don't count multiplicity?}  Much of the theory presented in [1] should still work if we don't count the multiplicity of the solutions to $P(z) = c$.  The biggest issue with this change would be in the application of Vieta's formula, since it does use multiplicity.
\item \textbf{Classify polynomials satisfying $\dagger^0$}.  This could be another starting point for  classifying the $\dagger$ condition.  Understanding when $\dagger^0$ is satisfied could help to understand when $\dagger^c$ is satisfied.
\item \textbf{Generalize the results presented here to other ``underlying'' kernel functions}.  There are other underlying kernel functions one could use other than $1 + z\overline{w}$.  However, changing the underlying kernel function will change the dagger conditions.  There are some kernel functions for which the approach presented here might still work, in particular, kernel functions of the form $1 + (z\overline{w})^n$, where $n$ is a positive integer.
\end{enumerate}

\section{Appendix: Newton-Girard Identities and the Vieta Formula}
Here we take a brief look at the Newton-Girard identities and Vieta's formula; both quintessential tools in this paper.  See [2] for a more in depth historical introduction.

\begin{thm}[Vieta's formula]
Suppose $P(z) = \sum\limits_{j=0}^na_jz^j$ satisfies $a_n\neq 0$. If $z_1,\dots,z_n$ are the roots (counting multiplicity) of $P$, then
\[\sum_{1\leq j_1<\cdots<j_k\leq n}z_{i_1}\cdots z_{i_k} = \frac{(-1)^ka_{n-k}}{a_n}\]
\end{thm}

\begin{proof}
By assumption we have that
\[\sum_{j=0}^na_jz^j = a_n\prod_{j=1}^n(z-z_j) = a_nz^n + a_n\sum_{k=1}^n\Big((-1)^k\sum_{1\leq j_1<\cdots<j_k\leq n}z_{i_1}\cdots z_{i_k}\Big)z^{n-k}\]
equating coefficients yields Vieta's formula.
\end{proof}

The Newton-Girard identities involve symmetric polynomials, so we start with a few definitions and notational conventions.  The $k$th power sum in $n$ variables is the polynomial $p_{k,n} = \sum\limits_{i=1}^nz_i^k$.  The elementary symmetric polynomials in $n$ variables are defined by
\[e_0 = 1,\quad e_1 = \sum_{i=1}^nz_i,\quad e_2 = \sum_{1\leq i<j\leq n}z_iz_j,\quad \dots,\quad e_n = z_1z_2\cdots z_n,\quad e_k = 0\quad\forall k>n\]
\begin{thm}[Newton-Girard identities]
Let $e_m$ denote the $m$th elementary symmetric poliynomial in $n$ variables and $P_m$ denote the $m$th power sum in $n$ variables.  We have that
\[\sum_{l=0}^{k-1}\Big((-1)^le_l P_{k-l}\Big) + (-1)^kke_k = 0\quad\forall k,n\in\mathbb{N}\]
\end{thm}

\begin{proof}
From the proof of Vieta's formula we have
\[\prod_{j=1}^n(z-z_j) = z^n + \sum_{l=1}^n\Big((-1)^l\sum_{1\leq j_1<\cdots<j_l\leq n}z_{i_1}\cdots z_{i_l}\Big)z^{n-l} = \sum_{l=0}^n(-1)^le_lz^{n-l}\]
where we now think of the $z_j$ as free variables.  Suppose $k = n$, set $z = z_h$ where $h\in\{1, \dots, k\}$ to obtain
\[0 = \sum_{l=0}^ke_lz_h^{k-l} = \sum_{l=0}^{k-1}\Big((-1)^le_l z_h^{k-l}\Big) + (-1)^ke_k\]
Now sum the right hand side over $h$ to obtain:
\[0 = \sum_h\left(\sum_{l=0}^{k-1}\Big((-1)^le_l z_h^{k-l}\Big) + (-1)^ke_k\right) = \sum_{l=0}^{k-1}\Big((-1)^le_l P_{k-l}\Big) + (-1)^kke_k\]
Demonstrating the identity for $k = n$.  The case $n<k$ follows from setting $k-n$ of the $z_h$ to $0$ and the case $k<n$ follows from setting $n-k$ of the $z_h$ to $0$.
\end{proof}

The first two Newton-Girard identities are:

\begin{enumerate}
\item For $k = 1$:  $P_1 - e_1 = 0$
\item For $k = 2$:  $P_2 - e_1P_1 + 2e_2 = 0 $
\end{enumerate}
Solving the second formula for $P_2$, we obtain $P_2 = e_1P_1 - 2e_2$.  Now let $n$ be the degree of some polynomial $R$ that satisfies $\dagger_2$ or $\ddagger_2$, and plug the $n$ roots (counting multiplicity) of $R$ into $P_2$.  The identity becomes $P_2 = -2e_2$ since $\dagger_2$ (or $\ddagger_2$) implies that $e_1 = P_1 = 0$.  By Vieta's formula, we have that $e_2 = \frac{a_{n-2}}{a_n}$.  So we may conclude by stating that
\[P_2 = \sum_{R(\zeta)=0}\zeta^2=\frac{-2a_{n-2}}{a_n}\]
Although we only made use of the first two identities here, the other identities will likely be of use when the underlying kernel function is changed.


\begin{thebibliography}{9}
\bibitem{alpay} 
Alpay, D., Jorgensen, P., Lewkowicz, I., \& Martziano, I. (2015). Infinite product representations for kernels and iterations of functions. In \textit{Recent advances in inverse scattering, Schur analysis and stochastic processes} (pp. 67-87). Birkhäuser, Cham.
 
\bibitem{funkhouser1930short} 
Funkhouser, H. G. (1930). A short account of the history of symmetric functions of roots of equations. The American mathematical monthly, 37(7), 357-365.

 
\bibitem{tipton} 
Tipton, J. E. (2016). Reproducing kernel Hilbert spaces and complex dynamics. The University of Iowa.
\end{thebibliography}
\end{document}